\documentclass[final]{siamltex}
\usepackage{a4} 
\usepackage{amsmath}
\usepackage{amssymb}
\usepackage{hyperref}
\usepackage{graphicx}
\usepackage{caption,subcaption}
\usepackage{multirow}
\usepackage{xstring}
\usepackage[utf8]{inputenc}
\usepackage{amstext,amsbsy,amsopn,eucal,enumerate}
\usepackage{tikz}
\usepackage{pgfplots}
\usepackage{tabularx}

\newcommand{\expn}{\operatorname{e}}

\newcommand{\beq}{\begin{equation}}
\newcommand{\eeq}{\end{equation}}
\newcommand {\mat}      [1] {\left[\begin{array}{#1}}
\newcommand {\rix}          {\end{array}\right]}
\newcommand {\smat}      [1] {\left[\begin{smallmatrix}{#1}}
\newcommand {\srix}          {\end{smallmatrix}\right]}
\newcommand {\s}      [1] {\begin{smallmatrix}{#1}}
\newcommand {\se}          {\end{smallmatrix}}

\newtheorem{remark}{Remark}

\begin{document}

\title{An $L^2_T$-error bound for time-limited balanced truncation}

\author{Martin Redmann\thanks{Weierstrass Institute for Applied Analysis and Stochastics, Mohrenstrasse 39, 10117 Berlin Germany  (Email: {\tt 
martin.redmann@wias-berlin.de})}. The author~gratefully acknowledge the support from the DFG through the research unit FOR2402.}

\maketitle

\begin{abstract}
Model order reduction (MOR) is often applied to spatially-discretized partial differential equations to reduce their order and hence decrease computational complexity. A reduced system can be obtained, e.g., by time-limited 
balanced truncation, a method that aims to construct an accurate reduced order model on a given finite time interval $[0, T]$. This particular balancing related MOR technique is studied in this paper. An $L^2_T$-error bound based 
on the truncated time-limited singular values is proved and is the main result of this paper.
The derived error bound converges (as $T\rightarrow \infty$) to the well-known $\mathcal H_\infty$-error bound of unrestricted balanced truncation, a scheme that is used to construct a good reduced system on the entire time line. 
The techniques within the proofs of this paper can also be applied to unrestricted balanced truncation so that a relatively short time domain proof of the $\mathcal H_\infty$-error bound is found here.
\end{abstract}



\setcounter{page}{1}

  \setlength{\parindent}{0pt}


\section{Introduction}\label{secintro}

Many phenomena in real life can be described by partial differential equations. Famous examples are the 
motion of viscous fluids, the description of water or sound waves and the distribution of heat. In order to solve these equations numerically it is required to discretize in time and space. 
Discretizing in space usually leads to large scale systems of ordinary differential equation which usually cause large computational effort. To overcome this burden, model order reduction (MOR) can be used to 
replace a high dimensional system by one of smaller order aiming to capture the main information of the original system.\smallskip

In this paper, we consider the following linear, time-invariant system:
\begin{align}\label{controlsystemoriginal}
\dot x_o(t)=A_o x_o(t)+B_o u(t),\quad x_o(0)=0,\quad y(t)=C_o x_o(t),
\end{align}
where $A_o\in\mathbb R^{n\times n}$ is assumed to be Hurwitz implying asymptotic stability of the state equation in (\ref{controlsystemoriginal}). $B_o\in\mathbb R^{n\times m}$ is the input and 
$C_o\in\mathbb R^{p\times n}$ the output matrix. Moreover, let the control $u$ be square integrable with respect to time, i.e., $\left\|u\right\|^2_{L^2_T}:=\int_0^T \left\|u(s)\right\|^2_2 ds<\infty$ for $T<\infty$.
We study a MOR technique that is called time-limited balanced truncation (BT). It was introduced in \cite{morGawJ90} with the goal of constructing an accurate reduced order model (ROM) on a
finite time interval $[0, T]$. The idea of balancing MOR schemes is to simultaneously diagonalize so-called Gramians in order to create a system in which the dominant reachable and observable states are 
the same. Then the states that only contribute very little to the system dynamics are truncated to obtain a ROM. For time-limited BT time-limited reachability and observability Gramians are aimed to be diagonalized. 
These are defined as follows
\begin{align}\label{TLBT_Gram}
	P_{T}:=\int_0^{T} \expn^{A_o s}B_oB_o^\top \expn^{A_o^\top s} ds,\quad  Q_{T}:=\int_0^{T} \expn^{A_o^\top s}C_o^\top C_o \expn^{A_os} ds
\end{align}
and it can be shown that they are the unique solutions to \begin{subequations}\label{TLBT_Lyap}
 \begin{align}\label{TLBT_LyapP}
A_o P_{T}+P_{T}A_o^\top+B_oB_o^\top-\expn^{A_oT}B_oB_o^\top\expn^{A_o^\top T}&=0,\\\label{TLBT_LyapQ}
A_o^\top Q_{T}+Q_{T}A_o+C_o^\top C_o-\expn^{A_o^\top T} C_o^\top C_o\expn^{A_o T}&=0.
\end{align}
\end{subequations}
Throughout this paper let us assume that system (\ref{controlsystemoriginal}) is completely reachable and observable which is equivalent to $P_T$ and $Q_T$ being positive definite, see \cite{antoulas}. In order to diagonalize 
$P_T$ and $Q_T$ a state space transformation is used. This basically means that the original matrices $(A_o, B_o, C_o)$ 
are replaced by $(A, B, C):= (S A_o S^{-1}, S B_o, C_o S^{-1})$, where $S$ is an invertible matrix. This transformation does not change the quantity of interest $y$ but it can be chosen such that the
Gramians of the transformed system are equal and diagonal, i.e.,  $S P_{T} S^\top=S^{-\top}Q_{T}S^{-1}=\Sigma_{T}=\diag(\sigma_{T, 1}, \ldots, \sigma_{T, n})$ with $\sigma_{T, 1}\geq \ldots \geq \sigma_{T, n}>0$. These 
diagonal entries are called time-limited singular values and are given as the square root of the eigenvalues of $P_T Q_T$. The balancing transformation can be derived through the Cholesky factorizations $P_{T}=L_PL_P^\top$, 
$Q_{T}=L_QL_Q^\top$, and the singular value decomposition $X\Sigma_{T}Y^\top=L_Q^\top L_P$. The matrix $S$ and its inverse are then given by
$S=\Sigma_{T}^{-\tfrac{1}{2}}X^\top L^\top_Q$ and $S^{-1}=L_PY\Sigma_{T}^{-\tfrac{1}{2}}$, see, e.g.,~\cite{antoulas}. Now, the ROM with state space dimension $r$ is obtained by selecting the left upper $r\times r$ block of $A$ 
and choosing the the first $r$ rows of $B$ as the input matrix as well as the first $r$ columns of $C$ as the output matrix.\smallskip

Unrestricted BT is a method that has already been widely studied \cite{antoulas, moore}. It relies on the infinite Gramians which are obtained by taking the limit $T\rightarrow \infty$ in (\ref{TLBT_Gram}). In \cite{BTstab}, 
the preservation of asymptotic stability in the ROM has been shown and in \cite{ennsbound,morGlo84} an $\mathcal H_\infty$-error bound was proved, moreover \cite{antoulas} contains an $\mathcal H_2$-error bound for unrestricted BT. 
\smallskip

Asymptotic stability is not preserved in the ROM for the time-limited case. However, error bounds exist such as $\mathcal{H}_2$-type error bounds that are quite recent. They can be found in \cite{kueduff, redmannkuerschner}.
An $\mathcal H_{\infty}$-error bound does not exist for the method considered here. However, there is one for a modified version of time-limited BT \cite{serkanantoulas}.
Time-limited BT for unstable systems is furthermore discussed in \cite{morKue18}. The main result of this paper is an $L^2_T$-error bound for time-limited BT that leads to the $\mathcal H_\infty$-bound 
in \cite{ennsbound,morGlo84} for $T\rightarrow \infty$. As a side effect a relatively short time domain 
proof of the bound in \cite{ennsbound,morGlo84} is presented which can be seen as a special case of the time-limited scenario. We conclude the paper by a numerical experiment in which the new error bound is tested.

\section{Reduced system and error bound for time-limited BT}\label{errorboundsBT}

In this section, we work with the balanced realization $(A, B, C)$ of (\ref{controlsystemoriginal}) introduced in Section \ref{secintro} through the balancing transformation $S$.
Thus, (\ref{TLBT_LyapP}) and (\ref{TLBT_LyapQ}) become
\begin{align}\label{balancedreach}
 A \Sigma_T+\Sigma_T A^\top&= -BB^\top +F_{T}F_{T}^\top,\\ \label{balancedobserve}
 A^\top \Sigma_T+\Sigma_T A &= -C^\top C + G_{T}^\top G_{T},
                                       \end{align}
i.e., $S P_T S^\top=S^{-\top} Q_T S^{-1}=\Sigma_T=\diag(\sigma_{T, 1}, \ldots, \sigma_{T, n})>0$, where $G_T:=C_o \expn^{A_oT}S^{-1}$ and $F_T:=S \expn^{A_oT}B_o$.
We partition the balanced coefficients of (\ref{controlsystemoriginal}) as follows:
\begin{align}\label{partmatrices}
 A=\smat{A}_{11}&{A}_{12}\\ 
{A}_{21}&{A}_{22}\srix,\;B=\smat B_1 \\ B_2\srix,\;C= \smat C_1 & C_2\srix,
            \end{align}
where $A_{11}\in \mathbb R^{r\times r}$, $B_1\in \mathbb R^{r\times m}$ and $C_1\in \mathbb R^{p\times r}$ etc.
Furthermore, we partition the state variable $x$ of the balanced realization and the time-limited Gramian \begin{align}\label{partgramstate}
               x(t)=\smat x_1(t) \\ x_2(t)\srix\text{ and }\Sigma_T=\smat \Sigma_{T, 1}& \\ & \Sigma_{T, 2}\srix,
                         \end{align}
where $x_1$ takes values in $\mathbb R^r$ ($x_2$ accordingly), $\Sigma_{T, 1}$ contains the large time-limited singular values and $\Sigma_{T, 2}$ the small ones. The ROM by time-limited BT then is 
\begin{subequations}\label{romstochstatebt}
\begin{align}\label{romstateeq}
             \dot x_r(t)&=A_{11}x_r(t)+B_1u(t),\\ 
    y_r(t)&=C_1x_r(t),
            \end{align}
            \end{subequations}
where $x_r(0)=0$. In the following, an $L^2_T$-error bound is proved. To do so, we define the variables \begin{align}\label{partxminusplus}
               x_-(t)=\smat x_1(t)-x_r(t) \\ x_2(t)\srix\text{ and } x_+(t)=\smat x_1(t)+x_r(t) \\ x_2(t)\srix.          
                           \end{align}
The system for $x_-$ is given by  
\begin{subequations}\label{xminus}
\begin{align}\label{statexminus}
             \dot x_-(t)&=Ax_-(t) + \smat 0 \\ h(t)\srix,\\ \label{xminusoutput}
    y_-(t)&=Cx_-(t)=Cx(t)-C_1 x_r(t)=y(t)-y_r(t),
            \end{align}
            \end{subequations}
where $h(t):=A_{21}x_r(t)+B_2 u(t)$. We derive (\ref{xminus}) by comparing the balanced system  (\ref{controlsystemoriginal}) 
with the reduced system (\ref{romstochstatebt}) using the partitions in (\ref{partmatrices}) and (\ref{partgramstate}).
The equation for $x_+$ is obtained in a similar manner. In comparison to (\ref{statexminus}), the sign for the compensation term $h$ is different and an additional control term appears:
\begin{align}\label{xplus}
             \dot x_+(t)=Ax_+(t)+2 B u(t) - \smat 0 \\ h(t)\srix.
            \end{align}
The proof of the error bound is simply based on applying the product rule in order to find suitable representations for $x_-^\top(t) \Sigma_T x_-(t)$ and $x_+^\top(t) \Sigma_T^{-1} x_+(t)$.
 These representations are then used to compute the desired bound. Deriving error bounds through the variables $x_-$ and $x_+$ has been done before in \cite{redstochbil,redmannstochbilspa}. 
 We start with a special case before we focus on the general one.
  \begin{lemma}\label{mainthm}
Let $\Sigma_{T, 2}=\sigma_T I$, $y$ be the output of the full model (\ref{controlsystemoriginal}) and $y_r$ be the output of the ROM (\ref{romstochstatebt}). Then, for $T>0$, 
we have \begin{align*}
 \left\|y-y_r\right\|_{L^2_{T}}\leq 2 \sigma_T  c_T \left\|u\right\|_{L^2_T},       
  \end{align*}
where $c_T=\expn^{0.5 \max\{\|G_T\Sigma_T^{-\frac{1}{2}}\|_2^2, \|F^\top_T\Sigma_T^{-\frac{1}{2}}\|_2^2\}T}$.
\end{lemma}
\begin{proof}
We observe that $x_-(0)=0$ due to the zero initial conditions of $x$ and $x_r$. 
Combining this fact with the product rule, we determine an estimate for  $x_-^\top(t) \Sigma_T x_-(t)$. Hence, inserting (\ref{statexminus}), we find
\begin{align*}
x_-^\top(t)\Sigma_T x_-(t)&=2 \int_0^t x_-^\top(s) \Sigma_T \dot x_-(s) ds\\ &=2 \int_0^t x_-^\top(s)\Sigma_T\left(Ax_-(s) + \smat 0 \\ h(s)\srix \right) ds\\
&=\int_0^t x_-^\top(s)(A^\top\Sigma_T+\Sigma_T A )x_-(s)ds + c_-(t)
\end{align*}
for $t\in[0, T]$, where $c_-(t):=2 \int_0^t x_-^\top(s)\Sigma_T \smat 0 \\ h(s)\srix ds=2 \int_0^t x_2^\top(s) \Sigma_{T, 2} h(s) ds$. 
The identity for $c_-$ is obtained by using the partitions of $\Sigma_T$ and $x_-$ in (\ref{partgramstate}) and (\ref{partxminusplus}), respectively.
We insert (\ref{balancedobserve}) into the above equation for $x_-^\top(t)\Sigma_T x_-(t)$ and take (\ref{xminusoutput}) into account. This leads to 
\begin{align*}
x_-^\top(t)\Sigma_T x_-(t)&= \int_0^t x_-^\top(s)(G_{T}^\top G_{T}-C^\top C)x_-(s)ds+c_-(t)\\
&=c_-(t)-\left\|y-y_r\right\|^2_{L^2_{t}}+\int_0^t x_-^\top(s) G_{T}^\top G_{T} x_-(s) ds.
\end{align*}
Since $x_-^\top(s) G_{T}^\top G_{T} x_-(s)=\|G_T\Sigma_T^{-\frac{1}{2}} \Sigma_T^{\frac{1}{2}}x_-(s)\|_2^2\leq \|G_T\Sigma_T^{-\frac{1}{2}}\|_2^2\; x_-^\top(s)\Sigma_T x_-(s)\leq  k\; x_-^\top(s)\Sigma_T x_-(s)$,
where $k:=\max\{\|G_T\Sigma_T^{-\frac{1}{2}}\|_2^2, \|F^\top_T\Sigma_T^{-\frac{1}{2}}\|_2^2\}$, Lemma \ref{gronwall} gives
\begin{align*}
x_-^\top(t)\Sigma_T x_-(t)&\leq c_-(t)- \|y-y_r\|_{L^2_{t}}^2 +\int_0^t (c_-(s) - \left\|y-y_r\right\|_{L^2_{s}}^2) k \expn^{k (t-s)} ds\\
& \leq c_-(t)- \|y-y_r\|_{L^2_{t}}^2 +\int_0^t c_-(s)  k \expn^{k (t-s)} ds.
\end{align*}
This implies that\begin{align}\nonumber
\left\|y-y_r\right\|_{L^2_{t}}^2 &\leq c_-(t) +\int_0^t c_-(s)  k \expn^{k (t-s)} ds\\ \label{firstbound}
&=\sigma_T^2 \left(c_+(t) +\int_0^t c_+(s)  k \expn^{k (t-s)} ds\right),
\end{align}
where $c_+(t):=2 \int_0^t x_2^\top(s) \Sigma^{-1}_{T, 2} h(s) ds$, exploiting that $\Sigma_{T, 2} =\sigma_T I$. We derive an upper bound for the expression $x_+^\top(t) \Sigma^{-1} x_+(t)$ 
to further analyze (\ref{firstbound}). Through (\ref{xplus}), it holds that
\begin{align}\label{insetobserveequation2}
x_+^\top(t)\Sigma_T^{-1} x_+(t)&=2 \int_0^t x_+^\top(s) \Sigma_T^{-1} \dot x_+(s) ds\\\nonumber
&=2  \int_0^t x_+^\top(s) \Sigma_T^{-1} (Ax_+(s)+2 B u(s) - \smat 0 \\ h(s)\srix) ds\\\nonumber
&=  \int_0^t x_+^\top(s)(A^\top\Sigma_T^{-1} + \Sigma_T^{-1} A)x_+(s) ds\\\nonumber
&\quad +\int_0^t x_+^\top(s) \Sigma_T^{-1} (4 B u(s) - 2\smat 0 \\ h(s)\srix) ds.
\end{align}
We multiply (\ref{balancedreach}) with $\Sigma_T^{-1}$ from the left and from the right and obtain 
\begin{align*}
 A^\top\Sigma_T^{-1}+\Sigma_T^{-1} A-\Sigma_T^{-1} F_{T}F_{ T}^\top\Sigma_T^{-1}= -\Sigma_T^{-1}BB^\top\Sigma_T^{-1},\end{align*}
which, by the Schur complement condition on definiteness, implies \begin{align}\label{schurposdef}
 \mat{cc} A^\top \Sigma_T^{-1} +\Sigma_T^{-1} A-\Sigma_T^{-1} F_{T}F_{T}^\top\Sigma_T^{-1} & \Sigma_T^{-1}B\\
 B^\top \Sigma_T^{-1}& -I\rix\leq 0.
                                       \end{align}
We multiply (\ref{schurposdef}) with $\smat x_+ \\ 2u\srix^\top$ from the left and with $\smat x_+\\ 2u\srix$ from the right. This leads to 
\begin{align*}
4 \left\|u\right\|_{2}^2+x_+^\top\Sigma_T^{-1} F_{T}F_{T}^\top\Sigma_T^{-1}x_+ \geq  x_+^\top\left(A^\top \Sigma_T^{-1}+\Sigma_T^{-1} A\right)x_++4x_+^\top\Sigma_T^{-1} Bu.
\end{align*}
Applying this result to inequality (\ref{insetobserveequation2}) gives 
\begin{align*}
x_+^\top(t)\Sigma_T^{-1} x_+(t) \leq & 4 \left\|u\right\|_{L^2_t}^2+\int_0^t \left\|F_{T}^\top\Sigma_T^{-1}x_+(s)\right\|_{2}^2 ds -c_+(t),
\end{align*}
exploiting that $x_+^\top(s)\Sigma_T^{-1} \smat 0 \\ h(s)\srix=x_2^\top(s)\Sigma_{T, 2}^{-1} h(s)$ using the partitions of $x_+$ and $\Sigma_T$. 
Since $\left\|F_{T}^\top\Sigma_T^{-1}x_+(s)\right\|_{2}^2\leq \left\|F_{T}^\top\Sigma_T^{-\frac{1}{2}}\right\|_{2}^2 x_+^\top(s)\Sigma_T^{-1} x_+(s)\leq k\;x_+^\top(s)\Sigma_T^{-1} x_+(s)$, we obtain
\begin{align}\label{insetobserveequationbla}
x_+^\top(t)\Sigma_T^{-1} x_+(t) \leq & 4 \left\|u\right\|_{L^2_t}^2-c_+(t)+k \int_0^t x_+^\top(s)\Sigma_T^{-1} x_+(s) ds.
\end{align}
Applying the Lemma of Gronwall (Lemma \ref{gronwall}) to (\ref{insetobserveequationbla}) yields
\begin{align}\label{estimatebla}
x_+^\top(t)\Sigma_T^{-1} x_+(t) &\leq 4 \left\|u\right\|_{L^2_t}^2-c_+(t) +\int_0^t (4 \left\|u\right\|_{L^2_s}^2-c_+(s)) k \expn^{k (t-s)} ds.
\end{align}
We then have that \begin{align}
\int_0^t \left\|u\right\|_{L^2_s}^2 k \expn^{k (t-s)} ds\leq \left\|u\right\|_{L^2_t}^2 \left[-\expn^{k (t-s)} \right]_{s=0}^t=\left\|u\right\|_{L^2_t}^2\left(\expn^{k t} -1\right). \label{keineahnung2}
\end{align}
We insert (\ref{keineahnung2}) into (\ref{estimatebla}) and get
\begin{align*}
c_+(t) +\int_0^t c_+(s) k \expn^{k (t-s)} ds\leq 4\left\|u\right\|_{L^2_t}^2\expn^{k t}.
\end{align*}
Comparing this result with (\ref{firstbound}) gives 
 \begin{align}\label{secondbound}
\left\|y-y_r\right\|_{L^2_{t}} \leq 2\sigma_T \expn^{0.5 k t} \left\|u\right\|_{L^2_t},
\end{align}
which concludes this proof.
\end{proof}
\begin{remark}\label{remark1}
Let $S$ be the balancing transformation, then $\Sigma_T^{-1}=S Q_T^{-1} S^\top = S^{-\top} P_T^{-1} S^{-1}$. $\|G_T\Sigma_T^{-\frac{1}{2}}\|_2^2$ and $\|F^\top_T\Sigma_T^{-\frac{1}{2}}\|_2^2$ can be expressed 
with the help of the matrices corresponding to the original system, since \begin{align*}
G_T\Sigma_T^{-\frac{1}{2}} \left(G_T\Sigma_T^{-\frac{1}{2}}\right)^\top &=C_o \expn^{A_oT}S^{-1} (S Q_T^{-1} S^\top) S^{-\top} \expn^{A^\top_oT} C_o^\top =C_o \expn^{A_oT} Q_T^{-\frac{1}{2}} \left(C_o\expn^{A_oT} Q_T^{-\frac{1}{2}}\right)^\top,\\
F^\top_T\Sigma_T^{-\frac{1}{2}} \left(F^\top_T\Sigma_T^{-\frac{1}{2}}\right)^\top &=B^\top_o \expn^{A^\top_oT}S^{\top} (S^{-\top} P_T^{-1} S^{-1}) S \expn^{A_oT} B_o =B^\top_o \expn^{A^\top_oT} P_T^{-\frac{1}{2}} \left(B^\top_o \expn^{A^\top_oT} P_T^{-\frac{1}{2}}\right)^\top.
                                                                                                     \end{align*}
Hence, the constant in Lemma \ref{mainthm} is $c_T=\expn^{0.5 \max\{\|C_o \expn^{A_oT} Q_T^{-\frac{1}{2}}\|_2^2, \|B^\top_o \expn^{A^\top_oT} P_T^{-\frac{1}{2}}\|_2^2\}T}$.
\end{remark}
Lemma \ref{mainthm} is now used to prove the main result of this paper. The idea is to remove the time-limited singular values step by step and apply the above lemma several times.
\begin{theorem}\label{errorboundthm}
Let $\tilde \sigma_{T, 1}, \tilde \sigma_{T, 2}, \ldots, \tilde \sigma_{T, \kappa}$ be the distinct diagonal entries of $\Sigma_{T,2}$, i.e., 
$\Sigma_{T, 2}=\diag(\sigma_{T, r+1},\dots, \sigma_{T, n})=\diag(\tilde \sigma_{T, 1},\dots, \tilde\sigma_{T, \kappa})$.
Moreover, let $y$ and $y_r$ be the outputs of the full model (\ref{controlsystemoriginal}) and the reduced system (\ref{romstochstatebt}), respectively,  with zero initial conditions. Then, for $T>0$, 
it holds that \begin{align*}
 \left\|y-y_r\right\|_{L^2_{T}}\leq 2\left( \tilde \sigma_{T, 1} c_{T, 1} + \tilde \sigma_{T, 2} c_{T, 2}+ \ldots + \tilde \sigma_{T, \kappa} c_{T, \kappa} \right)\left\|u\right\|_{L^2_T},       
  \end{align*}
where $c_{T, i}=\expn^{0.5 \max\{\|G_T\Sigma_T^{-\frac{1}{2}}\smat I_{r_{i+1}}\\0 \srix\|_2^2, \|F^\top_T\Sigma_T^{-\frac{1}{2}}\smat I_{r_{i+1}}\\0 \srix\|_2^2\}T}$. Here, $I_{r_{i+1}}$ is the identity matrix of dimension 
$r_{i+1}$ that is computed through $r_{i+1}=r_{i}+m(\tilde\sigma_{T, i})$ for $i=1, 2 \ldots, \kappa-1$ setting $r_1=r$, where $m(\tilde\sigma_{i})$ is the multiplicity of $\tilde\sigma_{i}$. Further, we have 
$c_{T, \kappa}=\expn^{0.5 \max\{\|G_T\Sigma_T^{-\frac{1}{2}}\|_2^2, \|F^\top_T\Sigma_T^{-\frac{1}{2}}\|_2^2\}T}$.
\end{theorem}
\begin{proof}
We apply Lemma \ref{mainthm} several times in order to prove this result. We use the triangle inequality to find a bound between the error of $y$ and $y_r$:\begin{align*}
  \left\|y-y_r\right\|_{L^2_T}\leq \left\|y-y_{r_\kappa}\right\|_{L^2_T}+\left\|y_{r_\kappa}-y_{r_{\kappa-1}}\right\|_{L^2_T}+\ldots+\left\|y_{r_2}-y_{r}\right\|_{L^2_T},        
  \end{align*}
where $y_{r_{i}}$ is the output of the ROM with dimension $r_i$. In the first error term, only $\tilde\sigma_{T,\kappa}$ is removed from the system. 
Hence, we can apply Lemma \ref{mainthm} which gives\begin{align*}             
  \left\|y-y_{r_\kappa}\right\|_{L^2_T}\leq 2 \tilde \sigma_{T, \kappa} c_{T, \kappa}\left\|u\right\|_{L^2_T}.
\end{align*}
We can apply Lemma \ref{mainthm} again for the error between $y_{r_\kappa}$ and $y_{r_{\kappa-1}}$. This is because only $\tilde\sigma_{r_{\kappa-1}}$ is removed. 
Moreover, the matrix equations for the ROM with dimension $r_\kappa$ has the same form as (\ref{balancedreach}) and (\ref{balancedobserve}). 
To see this, the left upper blocks of (\ref{balancedreach}) and (\ref{balancedobserve}) need to be selected. This delivers the same kind of equations with respective submatrices of $A, B, C$ and, in particular 
$(F_T, G_T, \Sigma_T)$ are replaced by $(\tilde F_T, \tilde G_T, \tilde\Sigma_T):=(\smat I_{r_\kappa} & 0 \srix F_T, G_T\smat I_{r_\kappa}\\0 \srix, \smat I_{r_\kappa} & 0 \srix \Sigma_T\smat I_{r_\kappa}\\0 \srix)$. 
Due to Lemma \ref{mainthm}, it follows that
\begin{align*}
\left\|y_{r_\kappa}-y_{r_{\kappa-1}}\right\|_{L^2_T} &\leq 2 \tilde\sigma_{T, r_{\kappa-1}} \expn^{0.5 \max\{\|\tilde G_T\tilde \Sigma_T^{-\frac{1}{2}}\|_2^2, \|\tilde F^\top_T \tilde \Sigma_T^{-\frac{1}{2}}\|_2^2\}T}\left\|u\right\|_{L^2_T}
\\ &= 2 \tilde\sigma_{T, r_{\kappa-1}} c_{T, \kappa-1} \left\|u\right\|_{L^2_T},
\end{align*}
since $\tilde G_T\tilde \Sigma_T^{-\frac{1}{2}}=G_T\smat I_{r_\kappa}\\0 \srix \smat I_{r_\kappa} & 0 \srix\Sigma_T^{-\frac{1}{2}}\smat I_{r_\kappa}\\0 \srix=G_T \Sigma_T^{-\frac{1}{2}}\smat I_{r_\kappa} & 0\\0&0 \srix \smat I_{r_\kappa}\\0 \srix
=G_T \Sigma_T^{-\frac{1}{2}} \smat I_{r_\kappa}\\0 \srix$ and 
$\tilde F^\top_T\tilde \Sigma_T^{-\frac{1}{2}}=F^\top_T\smat I_{r_\kappa}\\0 \srix \smat I_{r_\kappa} & 0 \srix\Sigma_T^{-\frac{1}{2}}\smat I_{r_\kappa}\\0 \srix=F^\top_T \Sigma_T^{-\frac{1}{2}}\smat I_{r_\kappa} & 0\\0&0 \srix \smat I_{r_\kappa}\\0 \srix
=F^\top_T \Sigma_T^{-\frac{1}{2}} \smat I_{r_\kappa}\\0 \srix$. 
Repeatedly applying the above arguments to the other error terms, the claim follows. 
\end{proof}\\
For the case of unrestricted BT, it holds that $F_T=0$ and $G_T=0$ in (\ref{balancedreach}) and (\ref{balancedobserve}). This leads to $c_{T, i}=1$ for $i=1, 2 \ldots, \kappa$ in Theorem \ref{errorboundthm} 
which is the bound proved in \cite{ennsbound, morGlo84}. 
Consequently, the techniques in the proofs of Lemma \ref{mainthm} and Theorem \ref{errorboundthm}  can also be used for a rather short time domain proof for the bound in \cite{ennsbound, morGlo84}. 
It is even shorter for $F_T=0$ and $G_T=0$ since Gronwall's lemma doesn't have to be applied.\smallskip

Moreover, we observe that $\|G_T\|_2^2$ and $\|F^\top_T\|_2^2$ decay exponentially for $A$ being Hurwitz. Hence, $c_{T, i}\rightarrow 1$ for $T\rightarrow \infty$ and for all $i=1, 2 \ldots, \kappa$. Consequently, we see that for 
sufficiently large $T$ the error bound is mainly characterized by the truncated time-limited singular values. If they are small, the error is expected to be small. Therefore, it makes sense to choose the reduced order dimension 
$r$ based on the truncated singular values for sufficiently large terminal times $T$. However, the bound in Theorem \ref{errorboundthm} requires to know the balancing transformation $S$ which is practically not computed. 
Therefore, we provide an upper bound for the result in Theorem \ref{errorboundthm} in the next corollary.
\begin{corollary}\label{boundcoro}
Under the assumptions of Theorem \ref{errorboundthm} and for $T>0$ we have \begin{align*}             
  \left\|y-y_r\right\|_{L^2_T}\leq 2 c_T\left( \tilde \sigma_{T, 1} + \tilde \sigma_{T, 2} + \ldots + \tilde \sigma_{T, \kappa} \right)\left\|u\right\|_{L^2_T},
\end{align*}
where $c_T=\expn^{0.5 \max\{\|C_o \expn^{A_oT} Q_T^{-\frac{1}{2}}\|_2^2, \|B^\top_o \expn^{A^\top_oT} P_T^{-\frac{1}{2}}\|_2^2\}T}$.
\end{corollary}
\begin{proof}
For $i=1, 2, \ldots, \kappa-1$ we have \begin{align*}
     \max\{\|G_T\Sigma_T^{-\frac{1}{2}}\smat I_{r_{i+1}}\\0 \srix\|_2^2, \|F^\top_T\Sigma_T^{-\frac{1}{2}}\smat I_{r_{i+1}}\\0 \srix\|_2^2\}\leq \max\{\|G_T\Sigma_T^{-\frac{1}{2}}\|_2^2, \|F^\top_T\Sigma_T^{-\frac{1}{2}}\|_2^2\}                               
                                       \end{align*}
and hence $c_{T, i}\leq c_{T, \kappa}=c_T$ using Remark \ref{remark1}.
\end{proof}\\
The constant $c_T$ can be determined practically. We will computed the above bound for an example in the next section.

\section{Numerical experiments}

We test the derived $L^2_T$-error bound stated in Corollary \ref{boundcoro} with an example that is taken from \\ \url{http://slicot.org/20-site/126-benchmark-examples-for-model-reduction}. The particular example is a heat equation in a thin rod, where the corresponding 
data can be found in the file ``heat-cont.mat''. Here, the state space dimension is $n=200$ and a system with a single input and a single output is considered. We apply time-limited BT to this example. 
Furthermore, we fix the final time to $T=12$ and choose two different normalized controls. These are $u_1= \tilde u_1/\left\|\tilde u_1\right\|_{L^2_T}$ and $u_1= \tilde u_2/\left\|\tilde u_2\right\|_{L^2_T}$, where 
$\tilde u_1(t)=\sin(2/5\, \pi t)$ and $\tilde u_2(t)=\cos(2 \pi t)\expn^{-t}$. The outputs of the original model (\ref{controlsystemoriginal}) and the ROM (\ref{romstochstatebt}) corresponding 
to $u_i$ are denoted by $y^i$ and $y^i_r$, respectively ($i=1,2$).

\begin{table}[ht]
\centering
\begin{tabular}{|c|c||c|c|}
\hline 
$r$&$\left\|y^1-y^1_r\right\|_{L^2_{T\textcolor{white}{\sum}}}$&$\left\|y^2-y_r^2\right\|_{L^2_T}$& $2c_T\sum_{i=r+1}^{200\textcolor{white}{\Psi}}\sigma_i$\\
\hline
\hline
$2$ & $2.91$e$-04$&$1.62$e$-04$& $4.68$e$-03$\\
$4$&$1.88$e$-05$&$1.90$e$-05$& $2.55$e$-04$\\
$6$&$2.07$e$-07$&$3.26$e$-07$& $4.13$e$-06$\\
$8$& $1.67$e$-08$ &$1.93$e$-08$&$2.56$e$-07$\\
\hline
\end{tabular}
\caption{${L^2_T}$-error time-limited BT and error bounds for different reduced order dimensions $r$;  $u=u_1, u_2$ and $T=12$.}\label{tablel2error}
\end{table}
One can see from Table \ref{tablel2error} that the error bound is relatively tight for the heat equation example and choice of $T$.

\section{Conclusions}
In this paper, we described the procedure of time-limited balanced truncation, a balancing related model order technique that is applied to find a good reduced system on a finite time interval $[0, T]$.
$\mathcal H_2$-type error bounds for this scheme have already been studied. However, 
no bound for the output error in $L^2_T$ existed so far. We closed this gap in this paper which is the main contribution here. The obtained $L^2_T$-error bound showed that the reduced order dimension can be found based on 
the truncated time-limited singular values for a sufficiently large terminal time $T$. Moreover, the bound converges to the $\mathcal H_\infty$-error bound of unrestricted (classical) balanced truncation such that 
the main result of this paper can be seen as an extension of this $\mathcal H_\infty$-error bound.


\appendix

\section{Gronwall lemma}

In this appendix, we state a version of Gronwall's lemma that we used throughout this paper.
\begin{lemma}[Gronwall lemma]\label{gronwall}
Let $T>0$, $z, \alpha: [0, T]\rightarrow \mathbb R$ be continuous functions and $\beta: [0, T]\rightarrow \mathbb R$ be a nonnegative continuous function. 
If \begin{align*}
    z(t)\leq \alpha(t)+\int_0^t \beta(s) z(s) ds,
   \end{align*}
then for all $t\in[0, T]$, it holds that \begin{align}\label{gronwallineq}
    z(t)\leq \alpha(t)+\int_0^t \alpha(s)\beta(s) \exp\left(\int_s^t \beta(w)dw\right) ds.
   \end{align}
\end{lemma}
   \begin{proof}
  The result is shown as in \cite[Proposition 2.1]{gronwalllemma}.
 \end{proof}

\bibliographystyle{plain}

\end{document}